\documentclass[10pt]{amsart}
\usepackage{stmaryrd}
\usepackage{textcomp}
\usepackage{enumerate}
\usepackage{setspace}
\usepackage{amssymb}
\usepackage{amsthm}
\usepackage{amsmath}
\usepackage{graphicx}
\usepackage{extarrows}
\usepackage{marvosym}
\usepackage{empheq}
\usepackage{latexsym}
\usepackage{endnotes}
\usepackage{fontenc}
\usepackage{color}
\usepackage{comment}
\usepackage{url}
\usepackage{float}

\usepackage[colorlinks=true, allcolors=blue]{hyperref}

\usepackage[left=1in,right=1in,top=1in,bottom=1in]{geometry}

\newtheorem{theorem}{Theorem}

\newtheorem{lemma}[theorem]{Lemma}
\newtheorem{claim}[theorem]{Claim}
\newtheorem*{claim*}{Claim}

\newtheorem{fact}[theorem]{Fact}

\theoremstyle{definition}
\newtheorem{definition}[theorem]{Definition}
\newtheorem*{definition*}{Definition}

\theoremstyle{remark}

\newtheorem*{remark*}{Remark}

\begin{document}

	\newcommand\marginal[1]{\marginpar{\raggedright\parindent=0pt\tiny #1}}
	
	\theoremstyle{remark}

	\newcommand{\RED}{\color{red}}
	
	\newcommand{\supp}{\mathrm{supp}}
	\def\eps{\varepsilon}
	\def\HH{\mathcal{H}}
	\def\E{\mathbb{E}}
    \def\D{\mathbb{D}}
	\def\C{\mathbb{C}}
	\def\R{\mathbb{R}}
	\def\Z{\mathbb{Z}}
	\def\N{\mathbb{N}}
	\def\PP{\mathbb{P}}
	\def\l{\lambda}
	\def\s{\sigma}
	\def\t{\theta}
	\def\a{\alpha}
	\def\la{\langle}
	\def\ra{\rangle}	
	\def\Xt{\widetilde{X}}
	\def\Pt{\widetilde{P}}
	\def\Var{\mathrm{Var}}
	\def\PV{\mathrm{PV}}
	\def\NV{\mathrm{NV}}
	\def\NN{\mathcal{N}}
	\def\CC{\mathcal{C}}
    \newcommand{\dH}{d_{\mathrm{H}}}
    \newcommand{\CL}{\mathbf{CL}}
    
	\def\cA{\mathcal{A}}
	\def\cQ{\mathcal{Q}}	
	\def\cC{\mathcal{C}}
	\def\F{\mathcal{F}}
	\def\tm{\tilde{\mu}}
	\def\ts{\tilde{\sigma}}
	\def\L{\Lambda}
	\def\Q{\mathcal{Q}}
	\renewcommand{\P}{\mathbb{P}}
	\renewcommand{\Re}{\operatorname{Re}}
	\newcommand{\Kt}{\widetilde{K}}
	\newcommand{\mt}{\widetilde{\mu}}
	\newcommand{\Int}{\operatorname{int}}
	\newcommand{\Ext}{\operatorname{ext}}
	\newcommand{\nut}{\widetilde{\nu}}
	\newcommand{\mes}{\mathrm{mes}}
	\newcommand{\Cov}{\operatorname{Cov}}
	\newcommand{\cE}{\mathcal{E}}
    
    \newcommand{\Lim}{\operatorname*{Lim}}

    \newcommand{\oren}[1]{{\color{blue}{\tt [OY: #1]}}}
    \newcommand{\marcus}[1]{{\color{red}{\tt [MM: #1]}}}

	\title{Dragon curves in Littlewood roots}

    \author[Marcus Michelen]{Marcus Michelen}
    \address{Department of Mathematics, Northwestern University}
    \email{michelen@northwestern.edu}

    \author[Oren Yakir]{Oren Yakir}
    \address{Department of Mathematics, Massachusetts Institute of Technology}
    \email{oren.yakir@gmail.com}

	\begin{abstract}
    
    A Littlewood polynomial is a polynomial whose coefficients lie in $\{- 1, +1\}$. While the majority of roots of a Littlewood polynomial of large degree are near the unit circle, numerical experiments suggest that when plotting the roots of \emph{all} Littlewood polynomials of a given large degree, striking fractal structures appear away from the unit circle. These fractals resemble the attractor of a certain iterated function system and are known as \emph{dragon curves}. In this note, we provide a rigorous explanation of this phenomenon, along with an analysis of a random variant, saying that such fractal behavior is typical.    
	\end{abstract}
	
	\maketitle

    \vspace{-10mm}
    \begin{figure}[H]
		 	\centering	\scalebox{0.1}{\includegraphics{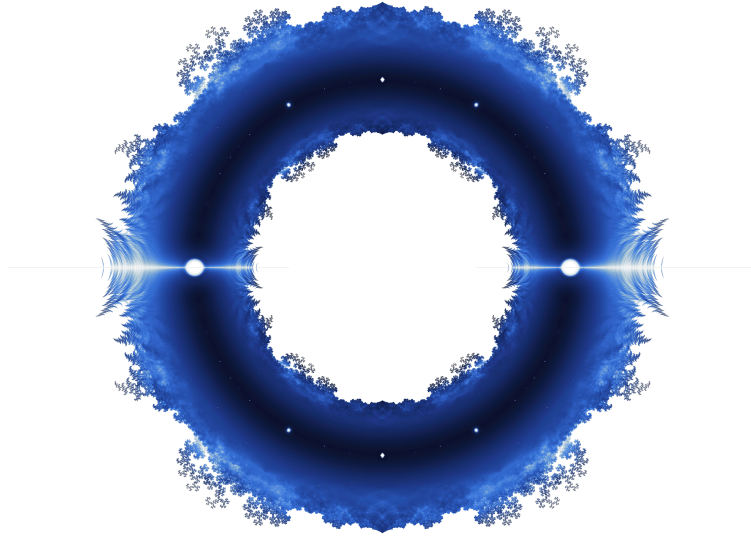}}
            \vspace{-2mm}
            \caption{Heat map for all the roots of polynomials with $\pm 1$ coefficients of degree $26$. }
            \label{figure:beauty_of_roots_with_random}
		\end{figure}
    
    \noindent
    A \textbf{Littlewood polynomial} is a polynomial with coefficients in $\{-1,+1\}$. We denote by $\mathcal{L}_n$ the set of all Littlewood polynomials of degree $n\ge 1$. Let $\D = \{z\in \C \, : \, |z|<1\}$ denote the open unit disk. In this paper, we study the (microscopic) structure of the set of all possible roots of Littlewood polynomials of large degree in the disk, namely 
    \begin{equation} \label{eq:def_of_Z_n}
        \mathcal{Z}_n = \Big\{ \alpha\in \D \, :\, f(\alpha) = 0 \ \text{for some } f\in \mathcal{L}_n\Big\} \, ,
    \end{equation}
    as $n\to \infty$. We note that whenever $\alpha$ is a root of the polynomial $f(z)\in \mathcal{L}_n$, then $\alpha^{-1}$ is a root of the polynomial $z^{n} f(z^{-1}) \in \mathcal{L}_n$. Hence, for our purposes, we lose no generality by restricting our attention to roots lying strictly within the unit disk $\D$.
    
    Similarly, a \textbf{Littlewood series} is a Taylor series centered at $z=0$ with coefficients in $\{-1,+1\}$, and we denote by $\mathcal{L}_\infty$ the set of all Littlewood series. Any Littlewood series defines an analytic function in $\D$ with unit radius of convergence. Let 
    \begin{equation*}
        \mathcal{Z}_\infty = \Big\{ \alpha\in \D \, :\, f(\alpha) = 0 \ \text{for some } f\in \mathcal{L}_\infty \Big\} \, .
    \end{equation*}
    Noting that for any $f \in \mathcal{L}_n$ we have $f(z)(1 - z^{n+1})^{-1} \in \mathcal{L}_\infty$ shows that $
    \bigcup_{n\ge 1} \mathcal{Z}_n \subset \mathcal{Z}_\infty $ and that this inclusion is strict\footnote{It is not hard to check that $1/2\in \mathcal{Z}_\infty$ but $1/2 \notin \mathcal{Z}_n$ for all $n\ge 1$.}. Hurwitz's theorem implies that the closure of this union 
    yields the full set of roots $\mathcal{Z}_\infty$.

    \noindent
    A heat map of the set $\mathcal{Z}_{26}$, shown in Figure~\ref{figure:beauty_of_roots_with_random} (inspired by a similar figure in~\cite{Beauty-of-roots}) reveals many interesting geometric features of this restricted set of roots. Some of these features are well understood. For example, the triangle inequality implies that $\mathcal{Z}_\infty \subset \{\frac{1}{2} \le |z| <1 \}$. Furthermore, a classical result of Erd\H{o}s and Tur\'an~\cite{Erdos-Turan-AOM} implies that most roots of a given Littlewood polynomial cluster around the unit circle as $n\to \infty$. This manifests in Figure~\ref{figure:beauty_of_roots_with_random} as an intensely dark neighborhood of the unit circle. It is also apparent that there are ``holes" in $\mathcal{Z}_{26}$ near roots of unity; see~\cite{Borwein_Pinner_1997} and~\cite[Section~4]{Konyagin-Schlag} for rigorous explanations. However, Bousch~\cite{Bousch} showed that
   $
    \big\{ 2^{-1/4} \le |z| < 1 \big\}\subset \mathcal{Z}_\infty \, ,
    $
    meaning that these holes get filled up as $n\to \infty$. 
    
    A striking, much less understood feature in Figure~\ref{figure:beauty_of_roots_with_random} is the appearance of fractal patterns for roots deep within the interior of the disk. A rigorous explanation for the appearance of these patterns is notably absent from the literature, a gap recently highlighted in the survey paper~\cite{Beauty-of-roots}. The main goal of this note is to bridge this gap. We show that if one has a Littlewood polynomial $P_n$ with a root $\alpha \in \D$ so that $|P_n'(\alpha)|$ is bounded below in terms of $n$ and $\alpha$, then the configuration of roots obtained from considering all Littlewood series with first $n$ coefficients given by $P_n$ is approximately a deterministic fractal (when rescaled and rotated) known as a \emph{dragon curve}.   We also prove that this lower bound on $|P_n'(\alpha)|$ is in fact generic by showing that such a lower bound occurs asymptotically almost-surely for all roots bounded away from the unit circle for random Littlewood power series.  We now properly define the limiting shapes known as dragon curves.

    \subsection*{Dragon curves}
    \label{subsec:what_is_a_dragon}
    For $\alpha\in \D$ we consider the pair of  functions
    \[
    \phi_{+}(w) = 1+\alpha w \, , \qquad \phi_{-}(w) = -1+\alpha w \, .
    \]
    Both maps $\phi_{\pm}:\C\to \C$ are contractions since $|\alpha| < 1$, 
    and by a classical theorem~\cite[Theorem~2.1.1]{Bishop-Peres-book}, there exists a unique non-empty compact set $D_\alpha\subset \C$ such that 
    \begin{equation*}
        D_\alpha = \phi_+(D_\alpha) \, \cup \, \phi_{-}(D_\alpha)\,  .
    \end{equation*}
    This attractor $D_\alpha$ is known as the \textbf{dragon curve} associated with the point $\alpha$. Some pictures of dragon curves are given in Figure~\ref{figure:dragon_curves} below.  In one sense, the connection between dragon curves and Littlewood series is pretty direct, as one can easily verify that
    \begin{equation} \label{eq:dragon_via_values_of_littlewood_series}
            D_\alpha = \big\{ f(\alpha) \, : \, f\in \mathcal{L}_\infty \big\} \, .
    \end{equation}
    We note that a simple self-similarity argument~\cite[Section~2.2]{Bishop-Peres-book} shows that the Hausdorff dimension of $D_\alpha$ is always bounded above by $\log(2)/\log(1/|\alpha|)$, which in particular implies that $D_\alpha$ has an empty interior for all $|\alpha| < 2^{-1/2}$. On the other hand, Shmerkin and Solomyak~\cite{shmerkin-solomyak} proved, among other things, that for almost every point in the annulus \ $\{2^{-1/2} < |\alpha| <1 \}$ the interior of $D_\alpha$ is non-empty.

    \begin{figure}[ht]
\centering
    \includegraphics[width=0.9\textwidth]{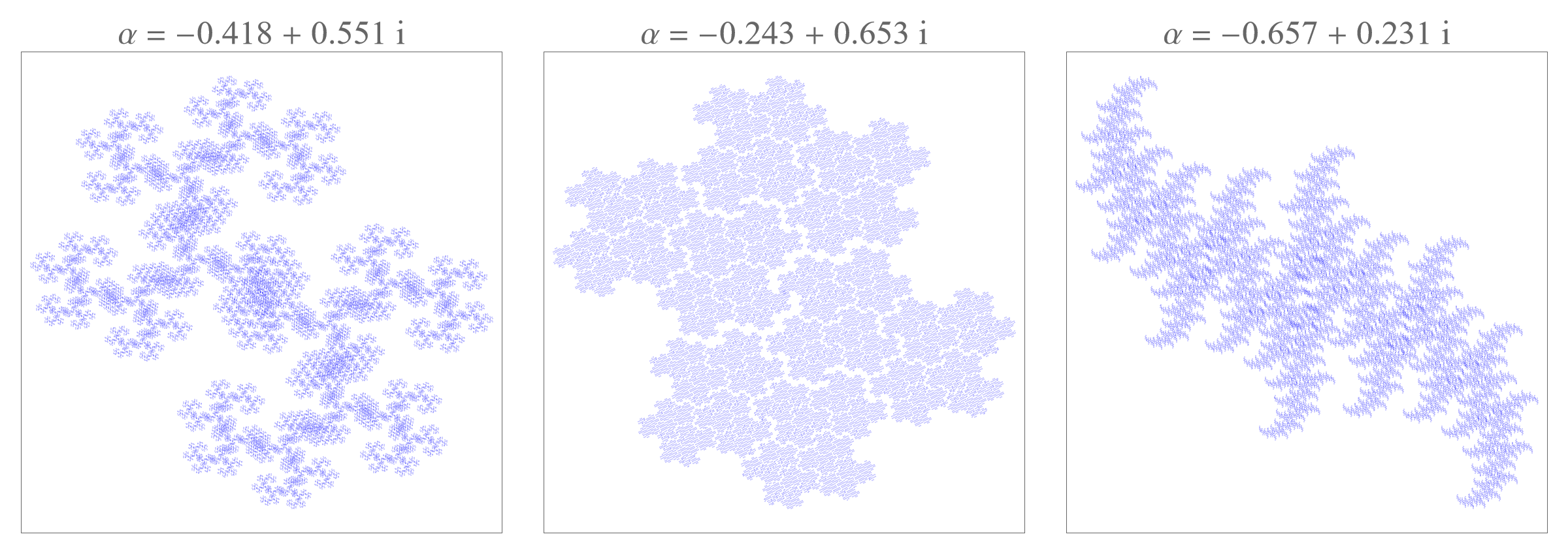}
    \\[2em]
    \includegraphics[width=0.27\textwidth]{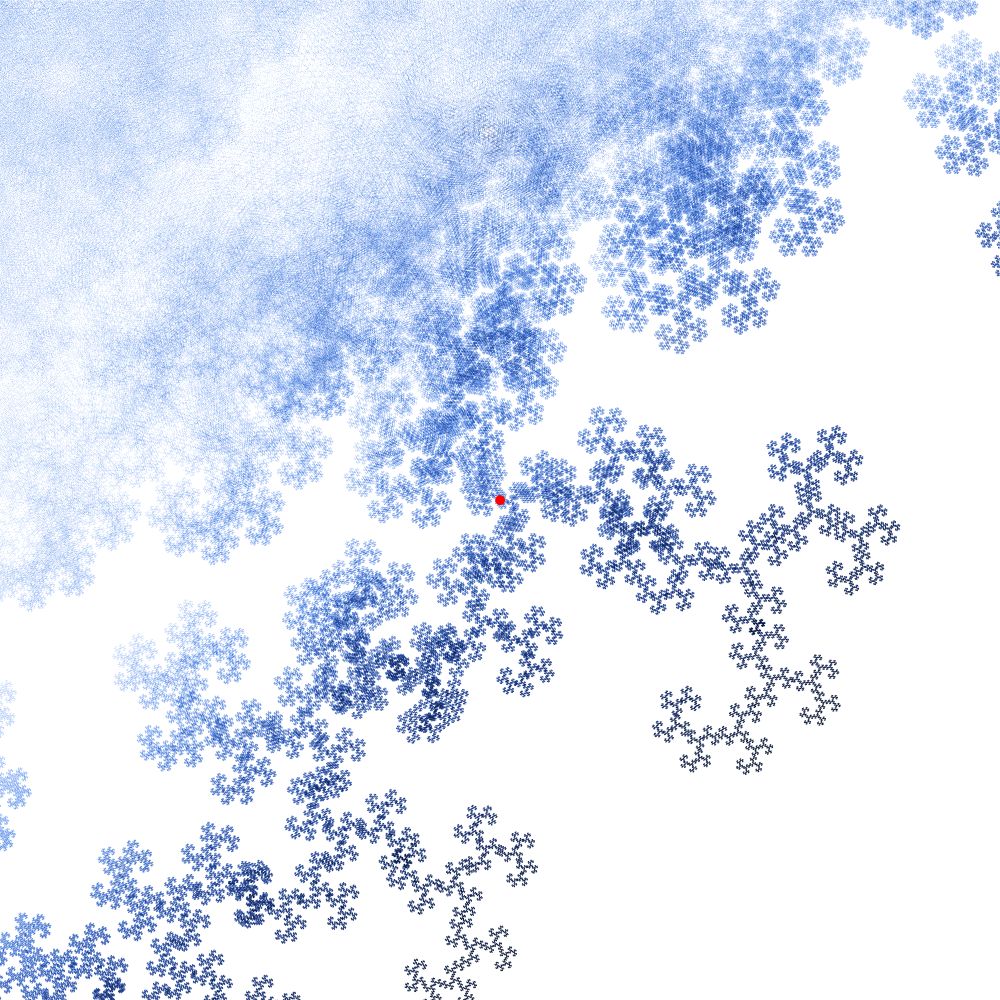}
    \hspace{1mm}
    \includegraphics[width=0.27\textwidth]{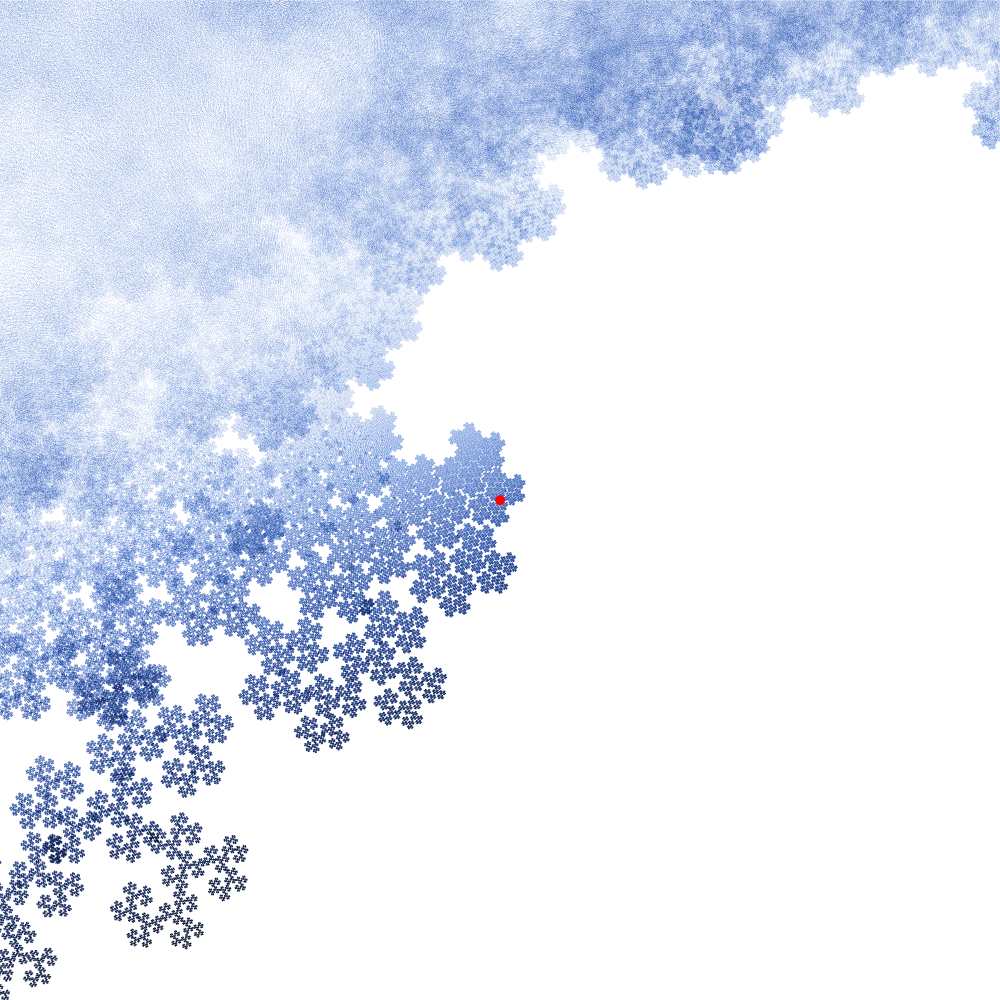}
    \hspace{1mm}
    \includegraphics[width=0.27\textwidth]{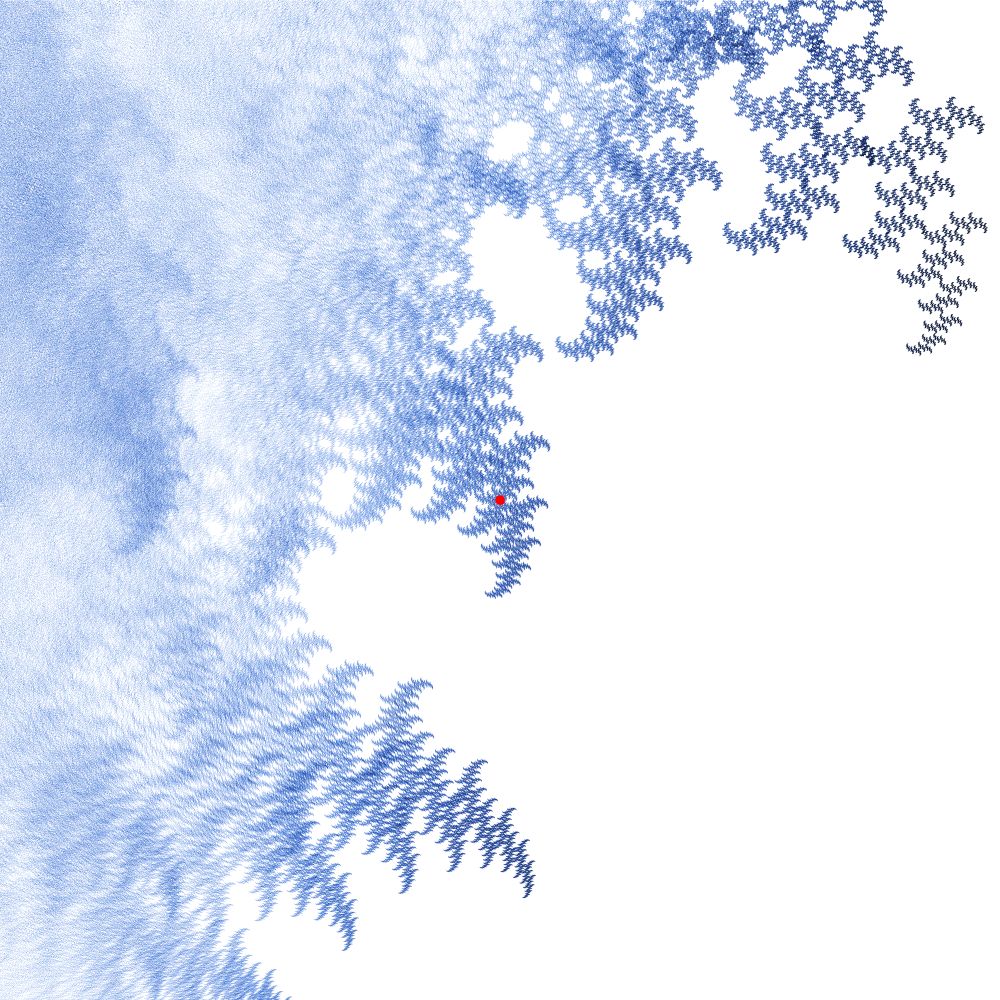}
 \caption{Top row: Dragon curves for three different choices of the parameter $\alpha$. \\ Bottom row: zoomed-in view near the same points (in red) from Figure~\ref{figure:beauty_of_roots_with_random}. 
 }
            \label{figure:dragon_curves}
\end{figure}

    \subsection*{Main deterministic theorem}
    Given a polynomial $P_n\in \mathcal{L}_n$ we will consider all possible ways to extend $P_n$ to a full Littlewood series. This collection is parametrized by $f\in \mathcal{L}_\infty$ via
    \[
    G_{P_n,f}(z) = P_{n}(z) + z^{n+1} f(z) \, . 
    \]
    We define the set of all possible extension roots as
    \begin{equation} \label{eq:def_of_R_F_n}
        R(P_n) = \Big\{ z\in \D \, : G_{P_n,f}(z) = 0 \ \text{for some } \, f\in \mathcal{L}_\infty \Big\} \, .
    \end{equation}
    We say that $P_n \in \mathcal{L}_n$ is the \emph{$n$th prefix} of $P \in \mathcal{L}_\infty$ if all terms up to degree $n$ match.  We note that if $P_n$ is the $n$th prefix of $P$ then $P_n \xrightarrow{n\to \infty} P$ uniformly on compact subsets of $\D$, and $\{R(P_n)\}_{n\ge 1}$ is a sequence of monotonically decreasing subsets of $\D$ (with respect to set inclusion) so that\footnote{The equality~\eqref{eq:roots_of_F_is_intersection_of_R_F_n} may also be understood as equality of multi-sets.} 
    \begin{equation} \label{eq:roots_of_F_is_intersection_of_R_F_n}
    \big\{\alpha\in \D \, : \, P(\alpha) = 0 \big\} = \bigcap_{n\ge 1}  R(P_n) \, .
    \end{equation}
    \setcounter{theorem}{-1}
    \begin{definition} \label{def:kappa_good}
        Let $K\subset \D$ be a compact set, and let $\kappa>0$. A Littlewood series $P\in \mathcal{L}_\infty$ is said to be \emph{$\kappa$-good} (with respect to $K$) if for any $\alpha\in K$ with $P(\alpha) = 0$ we have $|P^\prime(\alpha)| \ge \kappa$. 
    \end{definition}
    \noindent
    For $n\ge 1$ we denote the magnifying affine transformation around $\xi\in \D\setminus\{0\}$ by 
    \begin{equation} \label{eq:def_of_T_n_xi}
    T_{n,\xi} (z) = \frac{z-\xi}{\xi^{n+1}} \, .    
    \end{equation}
    Finally, recall that the Hausdorff distance between two non-empty subsets $A,B\subset \C$ is given by 
    \begin{equation} \label{eq:hausdorff_distance}
        \dH(A,B) = \max\bigg\{\sup_{a\in A} \, \text{dist}\big(a,B\big) \, , \,  \sup_{b\in B} \, \text{dist}\big(b,A\big) \bigg\} \, .
    \end{equation}
    Here, as always, for a point $z\in \C$ and a non-empty $A\subset \C$ we set 
    $
      \displaystyle  \text{dist}(z,A) = \inf_{a\in A} |z-a| \, . 
    $
    We are ready to state our main deterministic result. 
    \begin{theorem}
    \label{thm:main_deter_result}
        Let $K\subset \D$ be a compact set and let $\kappa>0$, then there exists $r_0>0$ so that the following holds. Let $P \in \mathcal{L}_\infty$ be a $\kappa$-good Littlewood series and let $P_n\in \mathcal{L}_n$ be its $n$th prefix.  Suppose $\alpha \in K$ so that $P(\alpha) = 0$.  Then there exists $\alpha_n \in K$ such that $P_n(\alpha_n) = 0$ and $\alpha_n \to \alpha$.
        If we define 
        \[
        R_{n,{\tt loc}} = R(P_n) \cap \D(\alpha,r_0)
        \]
        where $R(P_n)$ is given by~\eqref{eq:def_of_R_F_n}, then
        \[
        \lim_{n\to \infty} \dH\Big(T_{n,\alpha_n}\big(R_{n,{\tt loc}}\big) \, , \, \mathcal{D}_\alpha \Big) = 0
        \]
        where $\mathcal{D}_\alpha = P'(\alpha)^{-1} D_\alpha$, with $D_\alpha$ the dragon curve at $\alpha$ given by~\eqref{eq:dragon_via_values_of_littlewood_series}.
    \end{theorem}
    
    We note that our proof of Theorem~\ref{thm:main_deter_result} makes rigorous a heuristic explanation appearing in the literature (e.g.~\cite{Beauty-of-roots, Odlyzko-Poonen}), which, as far as we are aware, has not previously been turned into a mathematical proof.  
    \begin{remark*} 
    Our proof can be easily extended to show a version that does not explicitly require information about the limiting power series $P$.  In particular, given $\kappa > 0$, $\eps > 0$ and compact set $K \subset \D$, there is $n_0$ so that for all $n \geq n_0$ and for all Littlewoods  $P_n \in \mathcal{L}_n$ with root $\alpha_n \in K$ satisfying $|P_n'(\alpha_n)| \geq \kappa$,  $$\dH(T_{n,\alpha_n}(R_{n,{\tt loc}}),P_{n}'(\alpha_n)^{-1} D_{\alpha_n} ) \leq \eps \,. $$
    This may be easily deduced from Theorem \ref{thm:main_deter_result} since our proof will show that the assumption $|P_n'(\alpha_n)| \geq \kappa$ will imply that for all Littlewood power series $G_{P_n,f}$ with $f \in \mathcal{L}_\infty$ there is a nearby root $\alpha$ of $G_{P_n,f}$ with $|G_{P_n,f}'(\alpha)| = \big(1 + o(1)\big) |P_n'(\alpha_n)|$ as $n\to \infty$, where the $o(1)$ term is uniform in $f\in \mathcal{L}_\infty$.
    \end{remark*}
    
    To give a better feeling of what Theorem~\ref{thm:main_deter_result} says, we present a companion simulation in Figure~\ref{figure:illustration_of_main_result}. Consider the following three polynomials in $\mathcal{L}_{12}$:
    \begin{align*} 
            p_1(x) &= -1-x -x^2 + x^3  -x^4 -x^5-x^6-x^7 - x^8 -x^9 -x^{10} -x^{11} -x^{12} \, ,
\\
    p_2(x) &= -1-x -x^2 + x^3 +x^4 -x^5 - x^6 + x^7 + x^8 - x^9 - x^{10} + x^{11} - x^{12} \, ,
\\
   p_3(x) &= -1-x+x^2 + x^3 - x^4 - x^5 - x^6 + x^7 - x^8 + x^9 - x^{10} - x^{11} - x^{12} \, .
    \end{align*}        
    For each $j=1,2,3$, $p_j$ has a root at $\alpha_j$ where
    \[
    \alpha_1 \approx -0.418 + 0.551 \, i \, , \quad  \alpha_2\approx -0.243 + 0.653 \, i \, , \quad \alpha_3 \approx -0.657 + 0.231 \, i \, . 
    \]
    In Figure~\ref{figure:illustration_of_main_result}, we plot each of the roots of $p_j$ along with all roots of its extensions to degree-24 Littlewood polynomials (there are $2^{12}$ such extensions). We also plot a zoomed-in square near the root $\alpha_j$, which should be compared with the images of dragon curves from Figure~\ref{figure:dragon_curves}. 
    \begin{figure}
		 	\centering
            \scalebox{0.3}{\includegraphics{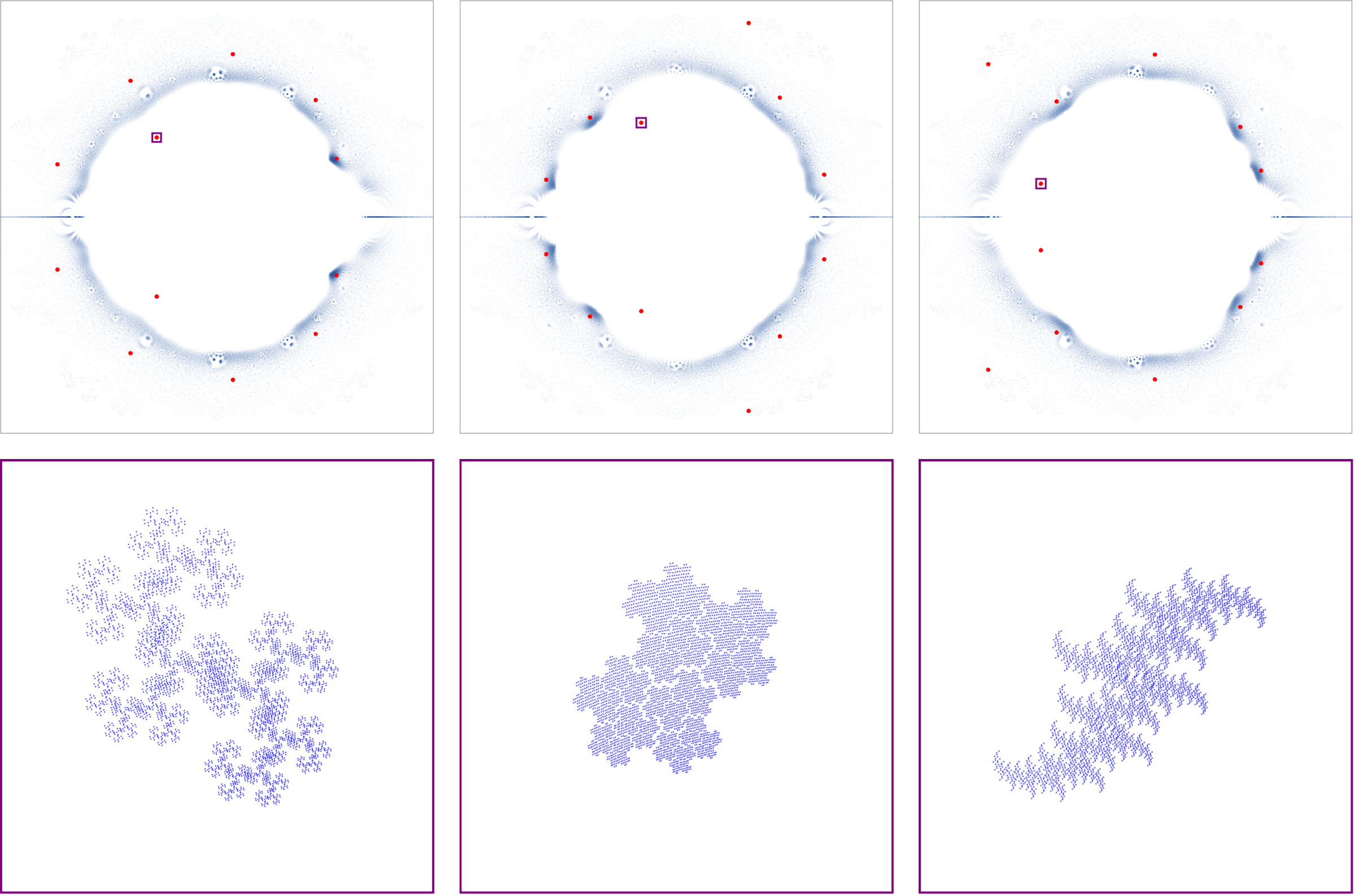}}
            \caption{Top row: roots of the degree-12 polynomials $p_1,p_2,p_3$ (red), together with the roots of all degree-24 Littlewood polynomials extending them (blue). \\ 
            Bottom row: zoomed-in view near the corresponding roots $\alpha_1,\alpha_2,\alpha_3$. Note that the roots resemble the dragons from Figure~\ref{figure:dragon_curves} above, but are multiplied by $p_j^\prime(\alpha_j)^{-1}$.}
            \label{figure:illustration_of_main_result}
		\end{figure}

    We remark that it is possible for a Littlewood polynomial (or series) to have a double root in $\D$. For instance, the Littlewood polynomial\footnote{We found this polynomial from the MathOverflow discussion at \url{https://mathoverflow.net/questions/424408/multiple-roots-of-polynomials-with-coefficients-pm-1}}
    \[
\begin{aligned}
&z^{27}+z^{26}+z^{25}-z^{24}-z^{23}-z^{22}+z^{21}+z^{20}+z^{19}+z^{18}
-z^{17}+z^{16}-z^{15}-z^{14}-z^{13}+z^{12} \\
&\qquad -z^{11}-z^{10}-z^{9}-z^{8}+z^{7}+z^{6}-z^{5}+z^{4}-z^{3}+z^{2}+z-1 \\
&=
\left(
z^{18}+z^{16}+2z^{13}-z^{12}+2z^{11}+z^{10}+3z^{8}+z^{6}
    +2z^{5}+2z^{3}+1
    \right)
    (z^{2}+1)(z-1)(z^{3}+z^{2}-1)^{2}
    \end{aligned}
    \]
    has a repeated factor $z^3+z^2-1$ with roots at $\approx 0.754$ and $
    \approx -0.877 \pm 0.744 \, i$. We will show in Theorem~\ref{thm:no_double_roots_almost_surely} below that such Littlewood polynomials are `rare' with respect to the uniform probability measure on $\mathcal{L}_n$.  This will verify the $\kappa$-good assumption for almost all $f \in \mathcal{L}_\infty$.
    
    \subsection*{The random model}
    There is a natural (uniform) probability measure one can place on $\mathcal{L}_\infty$ which is obtained by randomizing the coefficients. Indeed, let $(\eps_k)$ be a sequence of independent and identically distributed (i.i.d.) random variables such that 
    \[
    \P\big(\eps_k = 1\big) = \P\big(\eps_k = - 1\big) = \frac{1}{2} \, .
    \]
    We define the random Littlewood series $F$ by
    \begin{equation*}
        F(z) = \sum_{k=0}^\infty \eps_k z^k \in \mathcal{L}_\infty \,.
    \end{equation*}
    \begin{theorem}
    \label{thm:no_double_roots_almost_surely}
        Almost surely the random Littlewood series $F$ has no double zeros. In particular, almost surely for all compact $K\subset \D$ we have that $F$ is $\kappa$-good for all sufficiently small $\kappa>0$. 
    \end{theorem} 
    \noindent
    The result above shows that the conclusion of Theorem~\ref{thm:main_deter_result} can be applied to any (almost sure) realization of the random Taylor series $F$. We also note that Theorem~\ref{thm:no_double_roots_almost_surely} is a direct consequence of a more general theorem proved in our previous work~\cite[Theorem~1.3]{Michelen-Yakir-root-separation} which proves an analogous theorem for more general choices of random coefficients.  For ease of readability, we provide a  self-contained (and simplified) proof of Theorem~\ref{thm:no_double_roots_almost_surely} in Section~\ref{sec:proof_of_no_double_roots} below.
    
    To take a step further and discuss the global scaling limit for all root extensions in the disk $\D$---as opposed to just looking at the small disk of radius $r_0$ around a fixed root, as done in Theorem~\ref{thm:main_deter_result}---we need to introduce a coarser topology on closed sets than the one given by the Hausdorff metric.

    \subsection*{The Fell topology}
    Denote by $\CL(\C)$ the collection of all non-empty closed subsets of $\C$. When discussing global scaling limits for (random) root extensions (see Theorem~\ref{thm:tagged_point_process_convergence} below), we will require a notion of convergence for sequences in $\CL(\C)$. One natural candidate would be the Hausdorff metric topology~\eqref{eq:hausdorff_distance} which we used in Theorem~\ref{thm:main_deter_result}, but since the scaled set of roots can (and will) become unbounded as $n \to \infty$, we will require a more flexible notion of convergence.

    For a sequence $\{A_n\}\subset \CL(\C)$, we define the lower (inferior) and upper (superior) limits as
    \begin{align} \label{eq:def_of_limits_of_sets}
        \text{Li}(A_n) &= \Big\{ z\in \C \, : \, \limsup_{n\to \infty} \,  \text{dist}(z,A_n) = 0 \Big\} \, , \nonumber \\  \text{Ls}(A_n) &= \Big\{ z\in \C \, : \, \liminf_{n\to \infty} \,  \text{dist}(z,A_n) = 0 \Big\} \, .
    \end{align}
    It is straightforward to check that $\text{Li}(A_n) \subset \text{Ls}(A_n)$. If these sets coincide, their common value is defined as the \textbf{Kuratowski limit} of $\{A_n\}$, and we shall denote this limit by 
    \begin{equation} \label{eq:kuratowski_convergence}
    \Lim_{n\to \infty} A_n =  \text{Li}(A_n) = \text{Ls}(A_n) \, .    
    \end{equation}
    This convergence gives the topology on $\CL(\C)$, known as the \textbf{Fell topology}, see~\cite[Chapter~5]{Beer-topology-book}. The Fell topology is coarser than the standard Hausdorff metric topology, as it lets elements of $A_n$ escape to infinity without preventing convergence. For example, the sequence of sets $A_n = \{0,n\}$ converges in the Kuratowski sense to $\{0\}$, although $\dH(\{0\},\{0,n\}) = n$.
    When restricting to closed subsets contained (uniformly) inside a fixed compact domain in the plane, the Fell topology and the Hausdorff topology coincide.

    \subsection*{Point process convergence}
    For a sequence of random signs $(\eps_k)$ we consider 
    \begin{equation}
        \label{eq:reminder_of_F_n_and_F}
        F(z) = \sum_{k=0}^\infty \eps_k z^k \quad \text{and} \quad F_n(z) = \sum_{k=0}^n \eps_k z^k \, ,
    \end{equation}
    as random elements of $\mathcal{L}_\infty$ and $\mathcal{L}_n$, respectively. For $f\in \mathcal{L}_\infty$ we denote by
    \begin{equation*}
        G_{n,f}(z) = F_n(z) + z^{n+1} f(z) \, .
    \end{equation*}
    As before, the collection $\{G_{n,f}\}$ parametrizes all possible extensions of the random polynomial $F_n$ to the full Littlewood series.  We now consider the random closed set of all possible roots
    \begin{equation} \label{eq:def_of_R_n}
        R_n = R(F_n) =  \Big\{ z\in \D \, : \, G_{n,f}(z) = 0 \ \text{for some} \ f\in \mathcal{L}_\infty \Big\} \, .
    \end{equation}
    Recall that $T_{n,\alpha} (z) = \alpha^{-n-1} (z-\alpha)$. The next result shows that  clusters of roots in $\mathcal{Z}_\infty$ converge, in a probabilistic sense, to randomly scaled and rotated dragon curves. This gives a rigorous explanation as to why the dragon curves are visually present in Figure~\ref{figure:beauty_of_roots_with_random}. 
    \begin{theorem} \label{thm:tagged_point_process_convergence}
    For $F_n,F$ given by~\eqref{eq:reminder_of_F_n_and_F}, we have that
    \[
    \Big\{\big(\alpha, T_{n,\alpha}(R_n)\big) \, : \, \alpha\in \mathcal{Z}(F_n) \Big\} \xrightarrow{n\to \infty} \Big\{\big(\alpha, \mathcal{D}_\alpha\big) \, : \, \alpha\in \mathcal{Z}(F) \Big\}
    \]
    in distribution, with respect to the vague topology on (tagged) point processes on $\normalfont \D\times \textbf{CL}(\C)$.
    \end{theorem}
    \noindent
    In words, Theorem~\ref{thm:tagged_point_process_convergence} is saying that when zooming in around a typical point in Figure~\ref{figure:beauty_of_roots_with_random}, the local roots nearby arrange approximately along a scaled dragon curve centered at this point. Moreover, this description becomes increasingly accurate as the degree gets large. Additionally, note that Theorem \ref{thm:tagged_point_process_convergence} (and its deterministic counterpart Theorem \ref{thm:main_deter_result}) apply even when $|\alpha|$ is close to $1$.  One can understand these results as saying that the dragon curves still appear in this case, there are simply sufficiently many roots---and the dragons are full-dimensional---that they appear as a smearing near the unit circle in Figure \ref{figure:beauty_of_roots_with_random} rather than visually articulated fractals.
    
    In the course of the proof of Theorem~\ref{thm:tagged_point_process_convergence}, we establish a slightly stronger, almost sure, form of the convergence; see Section~\ref{sec:point_process_convergence} for details. Since convergence in distribution is the more standard formulation in the point process literature, we have chosen to state the theorem in this form.
    
    \subsection*{More related works} 
    The study of roots of Littlewood polynomials has a long history, going back at least to  Littlewood's famous monograph~\cite{Littlewood}. We do not attempt to give a full detailed account of this history; instead, we only mention the excellent book\footnote{In fact, a version of Figure~\ref{figure:beauty_of_roots_with_random} appears on the front cover of~\cite{Borwein-book}.} by Borwein~\cite{Borwein-book} that highlights some known techniques and many open problems. 

    Some other aspects of the (closure of the) 
    set of Littlewood roots $\mathcal{Z}_\infty$ have been studied before, mostly from the topological point of view. Bousch~\cite{Bousch} proved that $\mathcal{Z}_\infty$ is connected and locally connected. We also mention a related result by Odlyzko and Poonen~\cite{Odlyzko-Poonen}, which showed that the set of all roots of power series with $\{0,1\}$ coefficients is path-connected. Recently, Calegari, Koch and Walker~\cite[Theorem~11.3.1]{Calegari-Koch-Walker} showed there is a ``hole" in $\mathcal{Z}_\infty$, namely, a connected component of $\D\setminus \mathcal{Z}_\infty$ which is distinct from the obvious connected component centered at the origin. We also note that, from the dynamical systems point of view, it is slightly easier to study the set of all roots of power series with $\{-1,0,1\}$ coefficients, and indeed the study of topological properties of this set has a vast literature (see~\cite{Bandt, Barnsley_Harrington, Calegari-Koch-Walker, Solomyak_Xu}, to name a few). Finally, we also mention a recent paper by Hokken~\cite{Hokken_topology_of_zeros_square_disc}, showing that any point in $\mathcal{Z}_\infty$ can be approximated by a sequence of roots in $\mathcal{Z}_n$, for which the corresponding polynomial has a square discriminant.

    \subsubsection*{Notation} We denote by $\mathcal{L}_n$ the set of all Littlewood polynomials of degree $n\ge 1$, and by $\mathcal{L}_\infty$ the set of all Littlewood series. For $z\in \C$ and $r>0$, we always denote by $\D(z,r) = \{w\in \C \, : |w-z|<r\}$. We will also frequently use the notation $T_{n,\xi}(z) = (z-\xi)/\xi^{n+1}$, see~\eqref{eq:def_of_T_n_xi}. For $\alpha\in \D$ we always denote by $D_\alpha$ the dragon curve associated with the point $\alpha$, see~\eqref{eq:dragon_via_values_of_littlewood_series}. Finally, for an analytic function $g:\D\to \C$, we denote by $\mathcal{Z}(g)$ the (multi-)set of its roots in $\D$. Throughout, we denote by $C>0$ an arbitrary constant that may depend on some fixed non-asymptotic parameters, and may change from line to line. 
    
    \section{Local convergence to dragon curves}
    \label{sec:proof_of_local_convergence_deter}
    \noindent
    Throughout this section, we will denote by $P_n\in \mathcal{L}_n$ a Littlewood polynomial and by $P\in \mathcal{L}_\infty$ a series such that $P_n$ is always the $n$th prefix of $P$; 
    as such, $P_n\to P$ uniformly on compact subsets of $\D$ as $n\to\infty$. A simple application of Hurwitz's theorem (see e.g.~\cite[Chapter~5]{Ahlfors}) shows that convergence on compact sets is sufficient to deduce convergence of roots:
    \begin{fact} 
        \label{fact:almost_sure_convergence_of_roots}
        Let $K\subset \D$ be a compact set. Then for all $n$ large enough,  $P$ and $P_n$ have the same number of roots in $K$ provided $P$ has no roots on $\partial K$. Furthermore, if $\beta \in K$ is a root of $P$ with multiplicity $k\ge 1$, then for all $\eps>0$ small enough there exists $n$ large enough so that $P_n$ has exactly $k$ roots in $\D(\beta,\eps)$. 
    \end{fact}
    \noindent
    The next lemma (based on Rouch\'e's theorem) is the first main step towards the proof of Theorem~\ref{thm:main_deter_result}.
    \begin{lemma}
        \label{lemma:Rouche_for_littlewood_poly_and_remainder}
        Let $\kappa>0$ and $K\subset \D$ compact. Suppose that $\alpha\in K$ is such that $P(\alpha) = 0$ and $|P^\prime(\alpha)| \ge \kappa$. Then there exists $r_0>0$ so that for all sufficiently large $n$ and for all $f\in \mathcal{L}_\infty$ the function
        \[
        G_{P_n,f}(z) = P_n(z) + z^{n+1} f(z) 
        \]
        has a unique root in $\D(\alpha, r_0)$. Furthermore, denoting this root by $\beta_{n,f}$, we have that 
        \begin{equation} \label{eq:extenstion_root_and_original_root_are_exponentially_close}
           \sup_{f\in \mathcal{L}_\infty} |\alpha - \beta_{n,f}| \le \big(|\alpha|e^{1-|\alpha|}\big)^{n+1} \, .
        \end{equation}
    \end{lemma}
    \begin{proof}
        We will start the proof by showing~\eqref{eq:extenstion_root_and_original_root_are_exponentially_close}, that is, that there exists a root of $G_{n,f}$ in $$\D\Big(\alpha, \big(|\alpha|e^{1-|\alpha|}\big)^{n+1}\Big)$$ for all $n$ large enough. Indeed, let $r_0 = C^{-1} \kappa$ for some sufficiently large $C>0$. For any $\rho\in(0,r_0)$ and for any $z = \alpha + \eta$ with $|\eta| = \rho$ we have that
        \begin{equation} \label{eq:lower_bound_on_P_z_via_taylor}
            |P(z)| \ge \rho |P^\prime(\alpha)| - \rho^2 \max_{z: \mathrm{dist}(z,K) \leq \rho} |P^{\prime\prime}(z)| \ge \rho \kappa - C^\prime\rho^2 \ge \frac{\rho \kappa}{2} \, .
        \end{equation}
        We also have that
        \begin{equation*}
            \max_{|z-\alpha| = \rho} \big|G_{P_n,f}(z) - P(z)\big| \le \max_{|z-\alpha| = \rho} |z|^{n+1} |f(z)| \le \big(|\alpha| e^{(|\alpha| - 1)/2}\big)^{n+1}
        \end{equation*}
        for all $n$ large enough. Combining this with~\eqref{eq:lower_bound_on_P_z_via_taylor} applied with $\rho = \big(|\alpha| e^{(|\alpha| - 1)}\big)^{n+1}$ we get that for all $n$ large
        \[
        \max_{|z-\alpha| = \rho} \big|G_{P_n,f}(z) - P(z)\big| \le \frac{\kappa}{2} \big(|\alpha| e^{(|\alpha| - 1)}\big)^{n+1} \le \min_{|z-\alpha| = \rho} |P(z)| \, .
        \]
        Hence, by Rouch\'e's theorem, $P$ and $G_{P_n,f}$ have the same number of roots in $\D\big(\alpha, \big(|\alpha|e^{1-|\alpha|}\big)^{n+1}\big)$. By assumption we have that $P(\alpha) = 0$, so in particular there is a root $\beta_{n,f}$ of $G_{P_n,f}$ which satisfies~\eqref{eq:extenstion_root_and_original_root_are_exponentially_close}. To deduce uniqueness, we note that~\eqref{eq:lower_bound_on_P_z_via_taylor} also shows that $P$ has a unique root in $\D(\alpha,r_0)$ (precisely at $\alpha$), and  we observe that
        \[
        \max_{|z-\alpha| = r_0} \big|G_{P_n,f}(z) - P(z)\big| \le C \big(|\alpha| +r_0\big)^{n+1} \le \frac{\kappa r_0}{2} \le \min_{|z-\alpha| = r_0} |P(z)| \, ,  
        \]
        so another application of Rouch\'e's theorem shows that $P$ and $G_{P_n,f}$ have the same number of roots in $\D(\alpha,r_0)$, which means that $G_{P_n,f}$ has a unique root there. 
    \end{proof}
    \noindent    
    Before we turn to prove Theorem~\ref{thm:main_deter_result}, we record a simple claim that will be useful throughout.
    \begin{claim} 
        \label{claim:distance_between_littlewood_roots}
        
        Let $\alpha\in K$ be such that $P(\alpha) = 0$ and $|P^\prime(\alpha)|\ge \kappa$. Then for all $n$ large there exists $\alpha_n$ with $P_n(\alpha_n) = 0$ such that
        \[
        |\alpha_n - \alpha| \le C \kappa^{-1} |\alpha|^{n+1}
        \]
        where $C = C(K)>0$ depends only on the compact set $K\subset \D$.
    \end{claim}
    \begin{proof}
        Existence of a root $\alpha_n$ of $P_n$ in a small enough disk around $\alpha$ is guaranteed in view of Fact~\ref{fact:almost_sure_convergence_of_roots}. To estimate $|\alpha-\alpha_n|$, we Taylor expand and see that
        \[
        |P_n(\alpha)| \ge \big|P_n(\alpha_n) + (\alpha-\alpha_n)P_n^\prime(\alpha_n)\big| - \max_{z\in K} |P_n^{\prime\prime}(z)| |\alpha- \alpha_n|^2 \ge \frac{\kappa}{2} \, |\alpha - \alpha_n|
        \]
        for all $n$ large enough. On the other hand, we also have that
        \[
        |P_n(\alpha)| = |P_n(\alpha) - P(\alpha)| \le |\alpha|^{n+1} \sum_{k\ge 0} |\alpha|^k \le C\, |\alpha|^{n+1} \, ,
        \]
        and the claim follows. 
    \end{proof} 
    \begin{proof}[Proof of Theorem~\ref{thm:main_deter_result}]
        Recall that $R_{n,{\tt loc}} = R(P_n) \cap \D(\alpha,r_0)$ are the roots of all possible extensions $\{G_{P_n,f}\}_{f\in \mathcal{L}_\infty}$ which lie in a small neighborhood of $\alpha$. By Lemma~\ref{lemma:Rouche_for_littlewood_poly_and_remainder}, this set consists precisely of the roots $\beta_{n,f}$, which are the (unique) roots of $G_{P_n,f}$ which satisfy~\eqref{eq:extenstion_root_and_original_root_are_exponentially_close}. 
        We will show that
        \begin{equation} \label{eq:uniform_limit_of_scaled_points_is_dragon}
            \lim_{n\to\infty} \sup_{f\in \mathcal{L}_\infty} \Big| T_{n,\alpha_n}(\beta_{n,f}) + \frac{f(\alpha)}{P^\prime(\alpha)}\Big|  = 0
        \end{equation}
        which in turn will imply that $\displaystyle \lim_{n\to\infty} \dH\big(T_{n,\alpha_n}(R_{n,{\tt loc}}), \, \mathcal{D_\alpha}\big) = 0$. Indeed, letting $$\eps_n = \sup_{f\in \mathcal{L}_\infty} \Big| T_{n,\alpha_n}(\beta_{n,f}) + \frac{f(\alpha)}{P^\prime(\alpha)}\Big| \, ,$$
        we get that for an arbitrary point $P^\prime(\alpha)^{-1}f(\alpha)\in \mathcal{D}_\alpha$ there exists the root $\beta_{n,-f}\in R_{n,{\tt loc}}$ so that
        \[
        \text{dist}\big(P^\prime(\alpha)^{-1}f(\alpha), T_{n,\alpha_n}(R_{n , \tt loc})\big) \le \big|P^\prime(\alpha)^{-1}f(\alpha) - T_{n,\alpha_n}(\beta_{n,-f})\big| \le \eps_n \, .
        \]
        Furthermore, since any point in $R_{n,\tt loc}$ is a root $\beta_{n,f}$ for some $f\in \mathcal{L}_\infty$, similar reasoning shows that
        \[
        \text{dist}\big(T_{n,\alpha_n}(\beta_{n,f}), \mathcal{D}_\alpha\big) \le \eps_n \, .
        \]
        Altogether, we get $\dH(T_{n,\alpha_n}(R_{n,{\tt loc}}), \mathcal{D}_\alpha) \le \eps_n$, so the theorem will follow once~\eqref{eq:uniform_limit_of_scaled_points_is_dragon} is established. 

        For each $f\in \mathcal{L}_\infty$, the relation $G_{P_n,f}(\beta_{n,f}) = 0$ implies that
        \begin{equation} \label{eq:beta_n_being_a_root_meaning}
            P_n(\beta_{n,f}) = -\beta_{n,f}^{n+1} f(\beta_{n,f}) \, .
        \end{equation}
        By~\eqref{eq:extenstion_root_and_original_root_are_exponentially_close} and writing $\gamma = (|\alpha| e^{1-|\alpha|})^{n+1}$ we have that
        \[
        \beta_{n,f}^{n+1} = \alpha^{n+1}\big(1+O(n \gamma /|\alpha|)\big) = \alpha^{n+1}(1 + o(1))
        \]
        as $n\to \infty$, where the $o(1)$ term does not depend on $f\in \mathcal{L}_\infty$. Furthermore, we have that
        \[
        |f(\beta_{n,f}) - f(\alpha)| \le |\alpha - \beta_{n,f}| \cdot \max_{z\in K}|f^\prime(z)| \le C_K |\alpha-\beta_{n,f}| \xrightarrow{n\to \infty} 0 
        \]
        uniformly in $f$. Therefore, as $n\to \infty$, we may rewrite the right-hand side of \eqref{eq:beta_n_being_a_root_meaning} as \begin{equation}\label{eq:beta-in-terms-of-alpha}
            \beta_{n,f}^{n+1} f(\beta_{n,f}) = \alpha^{n+1} f(\alpha) + o(|\alpha|^{n+1})\,.
        \end{equation}
        On the other hand, by Claim~\ref{claim:distance_between_littlewood_roots} the sequence $\{\alpha_n\}$ with $P_n(\alpha_n) = 0$ must have $\alpha_n\in \D\big(\alpha, C \kappa^{-1} |\alpha|^{n+1}\big)$. Taylor expanding the left-hand side of~\eqref{eq:beta_n_being_a_root_meaning}, we get that
        \begin{align*}
            P_n (\beta_{n,f}) &= P_n(\alpha_n) + (\beta_{n,f} - \alpha_n) P_n^\prime(\alpha_n) + O\big(|\beta_{n,f} - \alpha_n|^2\big) \\ &= (\beta_{n,f} - \alpha_n)P_n^\prime(\alpha_n)\big(1 + o(1) \big)  = (\beta_{n,f} - \alpha_n)P^\prime(\alpha)\big(1 + o(1) \big) \, ,
        \end{align*}
        where the $o(1)$ term depends only on $K$ and $\kappa$. Plugging back into~\eqref{eq:beta_n_being_a_root_meaning} and rearranging gives
        \[
        T_{n,\alpha_n}(\beta_{n,f}) = -\frac{\beta_{n,f}-\alpha_n}{\alpha_n^{n+1}} = -\Big(\frac{\alpha}{\alpha_n}\Big)^{n+1} \frac{f(\alpha)}{P^\prime(\alpha)} + o(1)  = -\frac{f(\alpha)}{P^\prime(\alpha)} + o(1) \, ,
        \]
        with the $o(1)$ term being uniform in $f\in \mathcal{L}_\infty$ and in the last equality we used Claim~\ref{claim:distance_between_littlewood_roots}. This proves the limit~\eqref{eq:uniform_limit_of_scaled_points_is_dragon}, completing the proof.  
        \end{proof}
    
    \section{Almost sure no double zeros}
        \label{sec:proof_of_no_double_roots}
    \noindent
    Recall that $F(z) = \sum_{k=0}^\infty \eps_k z^k$ with $(\eps_k)$ an i.i.d.\ sequence of random signs. In this section we prove Theorem~\ref{thm:no_double_roots_almost_surely}, which states that almost surely $F$ has no double roots in the unit disk $\D$. 

    Fix some $r\in(\tfrac{1}{2},1)$, and recall that $r\D = \{ |z| \le r \}$. It is not hard to show (see Claim~\ref{claim:Jensen_bound_on_number_of_roots} below) that any power series with coefficients in $\{-1,+1\}$ has at most $C(r)<\infty$ roots (including multiplicities) in the disk $r\D$. For $j\ge 1$ and $\gamma>0$ we define the events
    \begin{equation}
        \label{def:event_E_r,j,gamma}
        \mathcal{E}_{r,j}(\gamma) = \Big\{ \exists z\in r\D \, : \, F(z) = F^\prime(z) = \ldots = F^{(j)}(z) = 0\, , \ \text{and} \ |F^{(j+1)}(z)| > \gamma  \Big\} \, .
    \end{equation}
    \begin{lemma}
        \label{lemma:E_r,j,gamma_has_prob_zero}
        For all $r<1$, $j\ge 1$ and $\gamma>0$ we have $\P\big(\mathcal{E}_{r,j}(\gamma)\big) = 0$.
    \end{lemma}
    
    \noindent
    Assuming Lemma~\ref{lemma:E_r,j,gamma_has_prob_zero} for a moment, it is quite simple to complete the proof of Theorem~\ref{thm:no_double_roots_almost_surely}.
    
    \begin{proof}[Proof of Theorem~\ref{thm:no_double_roots_almost_surely}]
        Since the events $\mathcal{E}_{r,j}(\gamma)$ are monotonically increasing as $\gamma\downarrow 0$, Lemma~\ref{lemma:E_r,j,gamma_has_prob_zero} implies 
        \begin{equation} \label{eq:E_r,j,0_has_zero_prob}
            \P\big(\mathcal{E}_{r,j}(0)\big) = \P \Big( \exists z\in r\D \, : \, F(z) = F^\prime(z) = \ldots = F^{(j)}(z) = 0\, , \ \text{and} \ |F^{(j+1)}(z)| > 0 \Big) = 0 \, .
        \end{equation}
        The union bound shows that 
        \begin{align*}
        \P\Big(\exists z\in r\D \, : \, F(z) = F^\prime(z) = 0 \Big) \le  \P\big(\mathcal{E}_{r,1}(0)\big)+ \P\Big(\exists z\in r\D \, : \, F(z) = F^\prime(z) = F^{\prime\prime}(z) = 0 \Big) \, .
        \end{align*}
        Iterating the above $C(r)$ times, we get that 
        \begin{align*}
            \P\Big(\exists z\in r\D \, : \, F(z) = F^\prime(z) = 0 \Big) \le \sum_{j=1}^{C(r)} \P\big(\mathcal{E}_{r,j}(0) \big) + \P\Big(\exists z\in r\D \, : \, F(z)  = \ldots = F^{(C(r)+1)}(z) = 0 \Big) \, .
        \end{align*}
        Combining~\eqref{eq:E_r,j,0_has_zero_prob} and Claim~\ref{claim:Jensen_bound_on_number_of_roots}, we see that all terms on the right-hand side vanish, and hence
        \[
        \P\Big(\exists z\in r\D \, : \, F(z) = F^\prime(z) = 0 \Big) = 0 \, .
        \]
        Finally, by taking a countable union of radii with $r\uparrow 1$, we conclude that 
        \[
        \P\Big(\exists z\in \D \, : \, F(z) = F^\prime(z) = 0 \Big) = 0
        \]
        which is what we wanted to show. 
    \end{proof}
    \noindent
    We devote the remaining portion of the section to proving Lemma~\ref{lemma:E_r,j,gamma_has_prob_zero}. For fixed $j\ge 1$ we note that 
    \begin{equation} \label{eq:formula_for_j_derivative_of_F}
        F^{(j)}(z) =  \sum_{k=j}^\infty \eps_k (k)_j z^{k-j} \, ,
    \end{equation}
    where $(k)_j = j!\binom{k}{j}$ is the falling factorial. 
    \begin{claim}
    \label{claim:Jensen_bound_on_number_of_roots}
        For all $j\ge 0$, $m\in\mathbb{N}\cup {\infty}$ and every choice of signs $h_k\in\{-1,+1\}$ the number of roots (including multiplicities) in $r\D$ of $$h(z) = \sum_{k=j}^m h_k  (k)_j \,  z^{k-j}$$ is at most $C(r,j)<\infty$. 
    \end{claim}
    \noindent
    \begin{proof}[Proof of Claim~\ref{claim:Jensen_bound_on_number_of_roots}]
        Recall Jensen's formula~\cite[Chapter~5]{Ahlfors} for analytic functions. For all $R<1$ we have
        \[
        \int_{0}^{2\pi} \log|h(Re^{i\theta})| \, \frac{{\rm d}\theta}{2\pi} - \log|h(0)| = \sum_{\substack{|\alpha|<R \\ h(\alpha) = 0}} \log\Big(\frac{R}{|\alpha|}\Big) \, .
        \]
        Applying the formula with $R = r+\tfrac{1}{2}(1-r)<1$ while using the fact that $\log|h(0)| = \log |j!| \geq 0$ gives
        \begin{align*}
            \#\big\{\alpha\in r\D \, : \, h(\alpha)= 0 \big\} & \le \Big(\log\Big(\frac{R}{r}\Big)\Big)^{-1} \sum_{\substack{|\alpha|<r \\ h(\alpha) = 0}} \log\Big(\frac{R}{|\alpha|}\Big) \\ &  \le  \Big(\log\Big(\frac{R}{r}\Big)\Big)^{-1} \int_{0}^{2\pi} \log|h(Re^{i\theta})| \, \frac{{\rm d}\theta}{2\pi} \le \Big(\log\Big(\frac{R}{r}\Big)\Big)^{-1} \log\Big|\sum_{k=j}^\infty k^j R^{k-j}\Big|  <\infty \, . \qedhere 
        \end{align*}
    \end{proof}
    \noindent
    For $m\ge 1$ we set 
    \begin{equation*}
        P_m(z) = \sum_{k=0}^m \eps_k z^k \, , \qquad T_m(z) = F(z) - P_m(z) \, .
    \end{equation*}
    The next lemma is the key step in our analysis. Roughly speaking, it says that if $F$ has a double root at $z_0\in r\D$ which is not a triple root, then for $m$ large there will be a nearby critical point of $P_m$ on which $F$ is atypically small. 
    \begin{lemma}
        \label{lemma:existence_of_bad_critical_point}
         Let $r\in(\tfrac{1}{2},1)$, $j\ge 1$ and $\gamma>0$ be fixed and let $m\ge m_0(j,r,\gamma)$ be large enough. On the event $\mathcal{E}_{r,j}(\gamma)$ there exists $w\in (e^{1-r}r)\D$ such that 
         \begin{equation*}
             P_m^{(j)}(w) = 0 \qquad \text{and} \qquad |F^{(j-1)}(w)| \le C m^{2j}\big(e^{2(1-r)}|w|\big)^{2m} \, ,
         \end{equation*} 
         for some $C = C(r,j,\gamma)>0$. 
    \end{lemma}
    \begin{proof}
        Let $z_0\in r\D$ be a point such that
        \begin{equation*}
            F^{(j-1)}(z_0) = F^{(j)}(z_0) = 0 \, , \qquad \text{and} \qquad |F^{(j+1)}(z_0)| > \gamma \, . 
        \end{equation*}
        On the event $\mathcal{E}_{r,j}(\gamma)$, we know that such a point exists. We would like to apply Rouch\'e's theorem to get a root of $P_m^{(j)}$ in a disk $\D(z_0,t)$ for some suitably chosen small $t>0$. Indeed, since $|\eps_k|\le 1$ for all $k\ge 0$, we have that
        \[
        \max_{\ell\le j+2} \max_{|z|\le r} |F^{(\ell)}(z)| \le \max_{\ell\le j+2} \sum_{k= \ell}^\infty k^\ell r^{k-\ell}  < \infty.
        \]
        At a point $z = z_0 + \eta$ with $|\eta| = t$, a Taylor expansion shows that
        \begin{equation} \label{eq:proof_of_lemma_exists_bad_critical_points_lower_bound_via_taylor}
            |F^{(j)}(z)| \ge |F^{(j)}(z_0) + F^{(j+1)}(z_0) \, \eta| - t^2 \max_{|z|\le e^{r-1}r} |F^{(j+2)}(z)| > t\gamma - t^2 C(r,j) \, .
        \end{equation}
        On the other hand, by taking $$t = \gamma^{-1} C^\prime m^j (e^{r-1}|z_0|)^m \, ,$$
        for sufficiently large $C^\prime = C^\prime(r,j)$ we get that for all $z= z_0 + \eta$
        \[
        |T_m^{(j)}(z)| \le \sum_{k=m+1}^\infty k^j (|z_0|+t)^{k-j} \le C (r,j) \, m^j (e^{r-1}|z_0|)^m < |F^{(j)}(z)| \, ,
        \]
        where the last inequality holds by~\eqref{eq:proof_of_lemma_exists_bad_critical_points_lower_bound_via_taylor} for all $m$ sufficiently large. Rouch\'e's theorem now guarantees a root $w\in \D(z_0,t)$ of $P_m^{(j)} = F^{(j)} - T_m^{(j)}$, since $F^{(j)}(z_0) = 0$. We note that $$\big| |w|- |z_0|\big|\le t \le |z_0|e^{r-1}$$ for all $m$ sufficiently large. Finally, since $F^{(j-1)}(z_0) = F^{(j)}(z_0) = 0$, another Taylor approximation shows that
        \[
        |F^{(j-1)}(w)| \le t^2 \max_{|z|\le re^{r-1}} |F^{(j+1)}(z)| \le C(r,j,\gamma) m^{2j} (e^{r-1}|z_0|)^{2m} \le C(r,j,\gamma) m^{2j} (e^{2(r-1)}|w|)^{2m}   \, , 
        \]
        as desired.
    \end{proof}
    \begin{proof}[Proof of Lemma~\ref{lemma:E_r,j,gamma_has_prob_zero}]
        To prove the lemma, we will make use of the classical solution of Erd\H{o}s~\cite{Erdos_small_ball_1945} to the Littlewood-Offord problem.
        For all $N\ge1$ and for any $c_1,\ldots, c_N\in \C$ with $|c_k|\ge 1$ we have
        \begin{equation}
            \label{eq:small_ball_bound}
            \sup_{\zeta \in \C} \, \P\Big( \Big| \sum_{k=1}^N\eps_k c_k - \zeta\Big| \le 1\Big) \le \frac{C}{\sqrt{N}} \, ,
        \end{equation}
        for some absolute constant $C>1$. We note that even weaker bounds, such as the one originally obtained by Littlewood and Offord~\cite{Littlewood_Offord} (with an extra $\log(N)$ factor on the right-hand side of~\eqref{eq:small_ball_bound}), would also suffice for our purpose. 
        From~\eqref{eq:small_ball_bound} we can deduce that for all $m\ge 1$, $\delta\in(0,1)$ and for all $z\in \D \setminus \{0\}$ we have 
        \begin{equation} \label{eq:anti_concentration_for_T_m_j}
            \max_{\zeta\in \C} \, \P\Big(\big|T_m^{(j)}(z) - \zeta \big| \le \delta |z|^{m} \Big) \le C  \bigg(\frac{\log(1/|z|)}{\log(1/\delta)}\bigg)^{1/2}. 
        \end{equation}
        for some absolute constant $C>1$. 
        Indeed, let $N = \lfloor \tfrac{1}{2} \log(1/\delta)/\log(1/|z|) \rfloor $ and observe that for $k\in \{m+1,\ldots, m+N\}$ we have 
        \begin{equation*}
            \delta^{-1} (k)_j |z|^{k-j-m} \ge \delta^{-1} |z|^N \ge 1 \, .
        \end{equation*}
        We get that 
        \begin{align*}
            \max_{\zeta\in \C} \, \P\Big(\big|T_m^{(j)}(z) - \zeta \big| \le \delta  |z|^{m} \Big) &= \max_{\zeta\in \C} \, \P\Big(\Big|\sum_{k=m+1}^\infty \eps_k (k)_j z^{k-j} - \zeta \Big| \le \delta  |z|^{m} \Big) \\ & \le \max_{\zeta\in \C} \, \P\Big(\Big|\sum_{k=m+1}^{m+N} \eps_k \big(\delta  |z|^{m}\big)^{-1} (k)_j z^{k-j} - \zeta \Big| \le 1 \Big)  \stackrel{\eqref{eq:small_ball_bound}}{\le} \frac{C}{\sqrt{N}} 
        \end{align*}
        which, by our choice of $N$, proves~\eqref{eq:anti_concentration_for_T_m_j}. 

        Denote by $\mathbf{Z}_{m,j,r}$ the (random) collection of roots of $P_m^{(j)}$ inside $(e^{r-1}r)\D$. With this notation, Lemma~\ref{lemma:existence_of_bad_critical_point} implies that, for all $m$ large enough,
        \begin{equation}
        \label{eq:proof_of_lemma_E_has_prob_zero_application_of_bad_critical_point}
            \P\big(\mathcal{E}_{r,j}(\gamma)\big) \le \P\Big(\exists w\in \mathbf{Z}_{m,j,r} \, : \,  |P_m^{(j-1)}(w) + T_m^{(j-1)}(w)| \le C(r,j) \,  m^{2j} (e^{2(r-1)}|w|)^{2m} \Big) \, .
        \end{equation}
        Let $\mathcal{F}_m = \sigma(\eps_0,\eps_1,\ldots,\eps_m)$ be the $\sigma$-algebra generated by the first $m+1$ random signs. Note that $P_m$ (and, hence, also $\mathbf{Z}_{m,j,r}$) is measurable with respect to $\mathcal{F}_m$, while $T_m$ is independent of it. Furthermore, by Claim~\ref{claim:Jensen_bound_on_number_of_roots} we know that $|\mathbf{Z}_{m,j,r}| \le C(r,j)$ for all realizations of $P_m$. Thus, we can apply the law of total expectation together with a union bound, and get that 
        \begin{multline*}
            \P\Big(\exists w\in \mathbf{Z}_{m,j,r} \, : \,  |P_m^{(j-1)}(w) + T_m^{(j-1)}(w)| \le C(r,j) \,  m^{2j} (e^{2(r-1)}|w|)^{2m} \Big) \\  = \E\Big[\P\Big(\exists w\in \mathbf{Z}_{m,j,r} \, : \,  |P_m^{(j-1)}(w) + T_m^{(j-1)}(w)| \le C(r,j) \,  m^{2j} (e^{2(r-1)}|w|)^{2m} \mid \mathcal{F}_m\Big)\Big] \\ \le \E\Big[\sum_{w\in \mathbf{Z}_{m,j,r}} \P\Big(|P_m^{(j-1)}(w) + T_m^{(j-1)}(w)| \le C(r,j) \,  m^{2j} (e^{2(r-1)}|w|)^{2m} \mid \mathcal{F}_m\Big)\Big] \\ \le C(r,j) \, \max_{|w| \le e^{r-1}r} \max_{\zeta\in \C}  \, \P\Big(|T_m^{(j-1)}(w) - \zeta| \le C(r,j) \,  m^{2j} (e^{2(r-1)}|w|)^{2m}\Big).  
        \end{multline*}
        It remains to note that the last term tends to $0$ as $m\to\infty$. Indeed, plugging into~\eqref{eq:anti_concentration_for_T_m_j} with $\delta = Cm^{2j} (e^{4(r-1)}|w|)^m$ while noting that $e^{4(r-1)}|w| \le e^{5(r-1)}r<1$, we get that
        \[
        \max_{|w| \le e^{r-1}r} \max_{\zeta\in \C}  \P\Big(|T_m^{(j-1)}(w) - \zeta| \le C(r,j) \,  m^{2j} (e^{2(r-1)}|w|)^{2m}\Big) \le C(r,j) \big(\log(1/\delta)\big)^{-1/2} \le \frac{C(r,j)}{\sqrt{m}} \, . 
        \]
        In view of~\eqref{eq:proof_of_lemma_E_has_prob_zero_application_of_bad_critical_point}, we get that 
        \[
        \P\big(\mathcal{E}_{r,j}(\gamma)\big) \le \frac{C(r,j)}{\sqrt{m}} 
        \]
        for all $m$ large enough, and hence $\P\big(\mathcal{E}_{r,j}(\gamma)\big) = 0$ as we wanted to show. 
    \end{proof}

    \section{Point process convergence}
    \label{sec:point_process_convergence}
    \noindent
    In this section we complete the remaining details in the proof of Theorem~\ref{thm:tagged_point_process_convergence}, which says that as $n\to\infty$ the random closed sets
    \[
    \Big\{\big(\alpha, T_{n,\alpha}(R_n)\big) \, : \, \alpha\in \mathcal{Z}(F_n) \Big\}  
    \]
    converge in distribution, with respect to the vague topology on point processes on $\D \times \textbf{CL}(\C)$, to
    \[
    \Big\{\big(\alpha, \mathcal{D}_\alpha\big) \, : \, \alpha\in \mathcal{Z}(F) \Big\} \, .
    \]
    Let $C_{c}(\D\times \mathbf{CL}(\C))$ denote the space of continuous and compactly supported functions $g:\D \times \textbf{CL}(\C) \to \R$. To check the above convergence, we need to show that for any $g\in C_{c}(\D\times \textbf{CL})$ we have that
    \begin{equation}
        \label{eq:vague_convergence_of_pp_what_we_need_to_show}
        \sum_{\alpha \in \mathcal{Z}(F_n)} g\big( \alpha, T_{n,\alpha}\big(R_n\big) \big) \xrightarrow[n\to \infty]{ \ d \ } \sum_{\alpha\in \mathcal{Z}(F)} g\big(\alpha, \mathcal{D}_\alpha\big)
    \end{equation}
    where the convergence is in distribution just for a sequence of random variables. In fact, since both $\{F_n\}$ and $F$ are coupled together by the same sequence of random coefficients, we will prove the stronger statement
    \begin{equation}
        \label{eq:almost_sure_convergence_for_pp_test_functions}
        \sum_{\alpha \in \mathcal{Z}(F_n)} g\big( \alpha, T_{n,\alpha}\big(R_n\big) \big) \xrightarrow[n\to \infty]{ \ a.s. \ } \sum_{\alpha\in \mathcal{Z}(F)} g\big(\alpha, \mathcal{D}_\alpha\big) \, ,
    \end{equation}
    where the almost-sure (a.s.) convergence here is with respect to the sequence of random coefficients $\{\eps_k\}$. Since almost-sure convergence implies convergence in distribution, we get that~\eqref{eq:vague_convergence_of_pp_what_we_need_to_show} follows, and it remains to show~\eqref{eq:almost_sure_convergence_for_pp_test_functions} for all $g\in C_c(\D\times \textbf{CL}(\C))$.

    We first deal with almost-sure convergence of the roots themselves.
    \begin{claim}
        \label{claim:convergence_of_roots_in_compacts}
        Let $F_n$ and $F$ be given by~\eqref{eq:reminder_of_F_n_and_F}. Then almost surely for each compact set $K\subset \D$ there exists a labeling of the roots
        \begin{equation*}
            K \cap \mathcal{Z}(F) = \big\{\alpha_1,\ldots,\alpha_m \big\} \qquad \text{and} 
            \qquad K \cap \mathcal{Z}(F_n) = \big\{\alpha_{1,n},\ldots,\alpha_{m,n} \big\} 
        \end{equation*}
        with $m<\infty$, so that for each $1\le j \le m$ we have
        $ \displaystyle       \lim_{n\to \infty} \alpha_{j,n} = \alpha_j \, .
        $
    \end{claim}
    \begin{proof}
        For each compact set $K\subset \D$ we have
        \[
        \sup_{z\in K}|F(z) - F_n(z)| \le \sup_{z\in K} \sum_{k=n+1}^\infty |z|^k \xrightarrow{n\to \infty} 0 \, .
        \]
        Furthermore, Theorem~\ref{thm:no_double_roots_almost_surely} implies that almost surely $F$ has no double roots in $K$, and the claim now follows from Hurwitz's theorem (Fact~\ref{fact:almost_sure_convergence_of_roots}). 
    \end{proof}
    \begin{proof}[Proof of Theorem~\ref{thm:tagged_point_process_convergence}]
        Let $g\in C_{c}(\D \times \textbf{CL}(\C))$. By the discussion in the beginning of the section, it is enough to show~\eqref{eq:almost_sure_convergence_for_pp_test_functions}. Denote by $\pi_1$ projection onto the first coordinate and set $K = \pi_1(\supp(g))$. Since $K\subset \D$ is a projection of a compact set, it is also compact. Claim~\ref{claim:convergence_of_roots_in_compacts} guarantees the almost-sure convergence of the roots of $F_n$ to the roots of $F$ in $K$, so the convergence~\eqref{eq:almost_sure_convergence_for_pp_test_functions} will follow once we show that for each $1\le j \le m$
        \begin{equation} \label{eq:convergence_of_labels_in_fell_topology}
            \Lim_{n\to \infty} T_{n,\alpha_{j,n}} (R_n) = \mathcal{D}_{\alpha_j}
        \end{equation}
        in the Kuratowski sense~\eqref{eq:kuratowski_convergence}. We also note that, by Claim~\ref{claim:Jensen_bound_on_number_of_roots}, $m$ is bounded by a (non-random) constant which depends only on $K\subset \D$. 
        Recall that $R_n$ is the random set~\eqref{eq:def_of_R_n} of possible root extensions, and that $\mathcal{D}_{\alpha_j}$ is the randomly rescaled dragon curve
        \begin{equation} \label{eq:recall_def_of_scaled_dragon}
            \mathcal{D}_{\alpha_j} = \frac{1}{F^\prime(\alpha_j)} D_{\alpha_j} = \Big\{ \frac{f(\alpha_j)}{F^\prime(\alpha_j)} \, : \, f\in \mathcal{L}_\infty \Big\} \, .
        \end{equation}
        Denote by
        \begin{equation} \label{eq:def_of_delta}
            \delta = \min_{1\le j\le m}  |F^\prime(\alpha_j)| \, .
        \end{equation}
        Theorem~\ref{thm:no_double_roots_almost_surely} implies that $\delta>0$ almost surely. In particular, almost surely $F$ is $\delta$-good in $K$, in the sense of Definition~\ref{def:kappa_good}. Theorem~\ref{thm:main_deter_result} then shows that there exists $r_0>0$ so that for each $1\le j \le m$ we have
        \begin{equation*}
            \lim_{n\to \infty} \dH\Big( T_{n,\alpha_{j,n}}\big(R_n \cap \D(\alpha_j,r_0)\big), \mathcal{D}_{\alpha_j} \Big) = 0 \, .
        \end{equation*}
        Since convergence in the Hausdorff metric topology implies convergence in the Fell topology (see~\cite[Chapter~5]{Beer-topology-book} for a simple proof), we get that 
        \begin{equation} \label{eq:dragon_contained_in_liminf}
            \mathcal{D}_{\alpha_j} \subset \text{Li} \, \Big(T_{n,\alpha_{j,n}}\big(R_n \cap \D(\alpha_j,r_0)\big) \Big) \subset \text{Li} \, \big(T_{n,\alpha_{j,n}}(R_n) \big) \, .
        \end{equation}
        \begin{claim}
        \label{claim:dragon_contains_limit_superior}
        Almost surely for all $j=1,\ldots,m$ we have $\normalfont \text{Ls}\big( T_{n,\alpha_{j,n}}(R_n) \big) \subset \mathcal{D}_{\alpha_j}  \, .$
        \end{claim}
    \noindent
    We provide the proof of Claim~\ref{claim:dragon_contains_limit_superior} below. To finish the proof of Theorem~\ref{thm:tagged_point_process_convergence}, we note that we always have $\text{Li} \, \big(T_{n,\alpha_{j,n}}(R_n) \big) \subset \text{Ls} \, \big(T_{n,\alpha_{j,n}}(R_n) \big)$, and by combining~\eqref{eq:dragon_contained_in_liminf} with Claim~\ref{claim:dragon_contains_limit_superior} we get that
    \[
    \mathcal{D}_{\alpha_j} =\text{Li} \, \big(T_{n,\alpha_{j,n}}(R_n) \big) = \text{Ls} \, \big(T_{n,\alpha_{j,n}}(R_n) \big) = \Lim_{n\to \infty} T_{n,\alpha_{j,n}}(R_n)
    \]
    in the sense of Kuratowski. This proves~\eqref{eq:convergence_of_labels_in_fell_topology}, and since $g:\D\times \textbf{CL}(\C)$ is continuous, we get that 
    \begin{equation*}
        \lim_{n\to\infty} \sum_{j=1}^m g\big(\alpha_{j,n},  T_{n,\alpha_{j,n}} (R_n)\big) = \sum_{j=1}^m g\big(\alpha_j, \mathcal{D}_{\alpha_j}\big) 
    \end{equation*}
    almost surely, 
    which proves~\eqref{eq:almost_sure_convergence_for_pp_test_functions}. As we explained in the beginning of this section, this implies the point process convergence of Theorem~\ref{thm:tagged_point_process_convergence} and we are done. 
    \end{proof}
    \begin{proof}[Proof of Claim~\ref{claim:dragon_contains_limit_superior}]
        To simplify notation, we drop the dependence on $1\le j \le m$, and show that for any sequence $\{\alpha_n\}\subset K$ such that $F_n(\alpha_n) = 0$ and $\alpha_n\to \alpha$ we have 
        \begin{equation} \label{eq:limsup_is_contained_in_dragon_no_j}
        \text{Ls}\big(T_{n,\alpha_n}(R_n)\big) \subset \mathcal{D}_\alpha \, .    
        \end{equation}
        We will use the ``hit-and-miss" criterion for the Fell topology: we have that~\eqref{eq:limsup_is_contained_in_dragon_no_j} holds if and only if for every compact set $\mathcal M\subset \C$ such that $\mathcal{M} \cap \mathcal{D}_\alpha = \emptyset$ there exists $n_0\in \N$ so that for all $n\ge n_0$
        \begin{equation*}
            \mathcal{M} \cap T_{n,\alpha_n}(R_n) = \emptyset \, .
        \end{equation*}
        By unpacking the definition~\eqref{eq:def_of_R_n}, this just means that for any $f\in \mathcal{L}_\infty$ and for any root $\xi$ of $G_{n,f}(z) = F_n(z) + z^{n+1} f(z)$ we have
        \begin{equation} \label{eq:every_root_of_G_n_f_avoids_M}
            T_{n,\alpha_n}(\xi) \not\in \mathcal{M}
        \end{equation}
        for all $n\ge n_0$. In view of Lemma~\ref{lemma:Rouche_for_littlewood_poly_and_remainder}, the proof of~\eqref{eq:every_root_of_G_n_f_avoids_M} splits into two cases: 
            \\[0.5em] 
        \underline{Case I:} Suppose first that the root $\xi$ is the unique root of $G_{n,f}$ in the disk $\D(\alpha, r_0)$. In the proof of Theorem~\ref{thm:main_deter_result}, we denoted this root by $\beta_{n,f}$, and showed that
        \[
        \lim_{n\to \infty} \sup_{f\in \mathcal{L}_\infty} \Big|T_{n,\alpha_n}(\beta_{n,f}) + \frac{f(\alpha)}{F^\prime(\alpha)}\Big| = 0 \, ,
        \]
        see~\eqref{eq:uniform_limit_of_scaled_points_is_dragon}. In particular, by~\eqref{eq:recall_def_of_scaled_dragon} the above display shows that
        \begin{equation} \label{eq:dist_between_scaled_close_root_and_dragon}
        \lim_{n\to \infty} \sup_{f\in \mathcal{L}_\infty} \, \text{dist}\big(T_{n,\alpha_n}(\xi) , \mathcal{D}_\alpha\big) = 0 \, .
        \end{equation}
        Since both $\mathcal{M}$ and $\mathcal{D}_\alpha$ are compact sets with $\mathcal{M} \cap \mathcal{D}_\alpha  = \emptyset$, we also have that $\text{dist}(\mathcal{M}, \mathcal{D}_\alpha)>0$. Combining this with~\eqref{eq:dist_between_scaled_close_root_and_dragon}, we get that~\eqref{eq:every_root_of_G_n_f_avoids_M} holds for all $n$ large enough. 
    \\[0.5em] 
    \underline{Case II:}
    Suppose now that $\xi$ is a root of $G_{n,f}$ such that $|\xi - \alpha| \ge r_0$, for $r_0>0$ given by Lemma~\ref{lemma:Rouche_for_littlewood_poly_and_remainder}. Note that $r_0$ depends on $\delta$ given by~\eqref{eq:def_of_delta}, but we still have almost surely that $r_0>0$, which is all that we shall use. Claim~\ref{claim:distance_between_littlewood_roots} implies that
    \[
    |\xi - \alpha_n| \ge |\xi -\alpha| - |\alpha - \alpha_n| \ge r_0 - C \delta^{-1} |\alpha|^{n+1} \ge r_0/2
    \]
    for all $n$ large enough. In particular, we get that
    \[
    |T_{n,\alpha_n}(\xi)| = |\alpha_n|^{-n-1} |\xi - \alpha_n| \ge \frac{1}{4} |\alpha|^{-n-1} r_0 \, .
    \]
    Since the right-hand side of the above display tends to infinity as $n\to \infty$, we see that for all $n$ large enough~\eqref{eq:every_root_of_G_n_f_avoids_M} holds, as $\mathcal{M}$ is a compact set. 
    
    All in all, we proved that~\eqref{eq:every_root_of_G_n_f_avoids_M} holds for every root of possible roots of $\{G_{n,f}\}_{f\in \mathcal{L}_\infty}$, which in turn proved the inclusion~\eqref{eq:limsup_is_contained_in_dragon_no_j} and we are done. 
    \end{proof}
   
    \subsection*{Acknowledgments} MM is supported in part by NSF grants DMS-2336788 and DMS-2246624. OY is supported in part by NSF postdoctoral fellowship DMS-2401136. The authors would also like to thank ICERM, where the first conversations leading to this work took place during the workshop \emph{Random polynomials and their applications} in August 2025.

\bibliographystyle{abbrv}
\bibliography{random_polynomials}

    \vspace{3mm}
\end{document}